\documentclass[12pt]{amsart}

\usepackage{mathrsfs,amssymb}

\usepackage{hyperref}

\usepackage[backend=bibtex8,maxbibnames=99,style=numeric]{biblatex}
\bibliography{referencias}

\AtEveryBibitem{%

  \ifentrytype{doi}{}{%
    \clearfield{doi}
    \clearfield{url}
  }
}

\renewbibmacro{in:}{}
\DeclareFieldFormat
  [article,inbook,incollection,inproceedings,patent,thesis,unpublished]
  {title}{#1\isdot}
\makeatletter
\def\url@leostyle{%
  \@ifundefined{selectfont}{\def\UrlFont{\sf}}{\def\UrlFont{\small\sffamily}}}
\makeatother
\usepackage{geometry}
\newtheorem{theorem}{Theorem}[section]
\newtheorem{lemma}[theorem]{Lemma}
\newtheorem{prop}[theorem]{Proposition}

\theoremstyle{definition}

\theoremstyle{remark}
\newtheorem{remark}[theorem]{Remark}

\newtheorem{example}[theorem]{Example}

\numberwithin{equation}{section}

\let \la=\lambda
\let \e=\varepsilon
\let \d=\delta
\let \o=\omega
\let \a=\alpha

\let \b=\beta

\let \O=\Omega
\let \si=\sigma

\let \ga=\gamma

\begin{document}
\title[Commutators of singular integrals revisited]
{Commutators of singular integrals revisited}

\author[A. K. Lerner]{Andrei K. Lerner}

\address[A. K. Lerner]{Department of Mathematics,
Bar-Ilan University, 5290002 Ramat Gan, Israel}
\email{lernera@math.biu.ac.il}

\author[S. Ombrosi]{Sheldy Ombrosi}
\address[S. Ombrosi]{Departamento de Matem\'atica\\
Universidad Nacional del Sur\\
Bah\'ia Blanca, 8000, Argentina}\email{sombrosi@uns.edu.ar}

\author[I. P. Rivera-R\'{\i}os]{Israel P. Rivera-R\'{\i}os}
\address[I. P. Rivera-R\'{\i}os]{ Department of Mathematics, University of the Basque Country and
BCAM, Basque Center for Applied Mathematics, Bilbao, Spain}
\email{petnapet@gmail.com}

\begin{abstract}
We obtain a Bloom-type characterization of the two-weighted boundedness of iterated commutators of singular integrals.
The necessity is established for a rather wide class of operators, providing a new result even in the unweighted setting for the first order commutators.
\end{abstract}

\keywords{Commutators, weighted inequalities, Calder\'on-Zygmund operators.}

\subjclass[2010]{42B20, 42B25}

\maketitle

\section{Introduction}

Given a linear operator $T$ and a locally integrable function $b$,
define the commutator $[b,T]$ of $T$ and $b$ by
\[
[b,T]f(x)=b(x)T(f)(x)-T(bf)(x).
\]
The iterated commutators $T_b^m, m\in {\mathbb N},$ are defined inductively by
$$T_b^mf=[b, T_b^{m-1}]f,\quad T_b^1f=[b,T]f.$$

We say that a linear operator $T$ is an
$\omega$-Calder\'on-Zygmund operator on ${\mathbb R}^n$ if $T$ is $L^2$ bounded, and can be represented as
\begin{equation}\label{repr}
Tf(x)=\int_{{\mathbb R}^n}K(x,y)f(y)dy\quad\text{for all}\,\,x\not\in \text{supp}\,f,
\end{equation}
with kernel $K$ satisfying the size condition
$|K(x,y)|\le \frac{C_K}{|x-y|^n},x\not=y,$ and the smoothness condition
$$|K(x,y)-K(x',y)|+|K(y,x)-K(y,x')|\le \o\left(\frac{|x-x'|}{|x-y|}\right)\frac{1}{|x-y|^n}$$
for $|x-y|>2|x-x'|$, where $\omega:[0,1]\to [0,\infty)$ is continuous, increasing, subadditive and $\o(0)=0$.
We say that $\o$ satisfies the Dini condition if $\int_0^1\omega(t)\frac{dt}{t}<\infty.$

In this paper, we will prove the following result.

\begin{theorem}\label{mainres}
Let $\mu,\lambda\in A_{p}$, $1<p<\infty$. Further, let $\nu=\left(\frac{\mu}{\lambda}\right)^{\frac{1}{p}}$ and $m\in {\mathbb N}$.
\begin{enumerate}
\renewcommand{\labelenumi}{(\roman{enumi})}
\item If $b\in BMO_{\nu^{1/m}}$, then for every $\omega$-Calder\'on-Zygmund
operator $T$ on ${\mathbb R}^n$ with $\omega$ satisfying the Dini condition,
\begin{equation}\label{itcase}
\|T_{b}^{m}f\|_{L^{p}(\lambda)}\leq c_{n,m,T}\|b\|_{BMO_{\nu^{1/m}}}^{m}\left([\lambda]_{A_{p}}[\mu]_{A_{p}}\right)^{\frac{m+1}{2}\max\left\{ 1,\frac{1}{p-1}\right\} }\|f\|_{L^{p}(\mu)}.
\end{equation}
\item Let $T_{\Omega}$ be an operator defined by (\ref{repr}) with $K(x,y)=\Omega\Big(\frac{x-y}{|x-y|}\Big)\frac{1}{|x-y|^n}$, where
$\O$ is a measurable function on $S^{n-1}$, which does not change sign and is not equivalent to zero on some open subset from  $S^{n-1}$. If there is $c>0$ such that for every bounded measurable set $E\subset {\mathbb R}^n$,
$$\|(T_{\O})_b^m(\chi_E)\|_{L^p(\la)}\le c\mu(E)^{1/p},$$
then $b\in BMO_{\nu^{1/m}}$.
\end{enumerate}
\end{theorem}

\begin{remark}\label{expl}
We emphasize that in part (ii) of Theorem \ref{mainres}, no size and regularity assumptions on $\O$ are imposed.
It will be useful, however, to distinguish a class of operators satisfying both parts of the theorem. Assume that
$$T_{\O}f(x)=\text{p.v.}\int_{{\mathbb R}^n}f(x-y)\frac{\O(y/|y|)}{|y|^n}dy,$$
where $\O$ is continuous on $S^{n-1}$, not identically zero and $\int_{S^{n-1}}\O\,d\si=0$.
Assuming additionally that
$$\o(\d)=\sup_{|\theta-\theta'|\le \d}|\Omega(\theta)-\Omega(\theta')|$$
satisfies the Dini condition, we obtain that $T_{\O}$ satisfies both parts of Theorem \ref{mainres}.
\end{remark}

Recall that $b\in BMO_{\eta}$ (for a given weight $\eta$) if
\[
\|b\|_{BMO_{\eta}}=\sup_{Q}\frac{1}{\eta(Q)}\int_{Q}|b(x)-b_{Q}|dx<\infty,
\]
where the supremum is taken over all cubes $Q\subset {\mathbb R}^n$.
Here we use the standard notations $\eta(Q)=\int_Q\eta$ and $b_Q=\frac{1}{|Q|}\int_Qb$. We also recall that $w\in A_p ,1<p<\infty,$ if
$$[w]_{A_p}=\sup_{Q\subset {\mathbb R}^n}\left(\frac{1}{|Q|}\int_Qw\right)\left(\frac{1}{|Q|}\int_Qw^{-\frac{1}{p-1}}\right)^{p-1}<\infty.$$

In what follows, we present a brief history preceding Theorem \ref{mainres}, and, in parallel, we outline our novel points.

\begin{list}{\labelitemi}{\leftmargin=1em}
\item Assume first that $m=1$ and $\la=\mu\equiv 1.$ In this case Theorem \ref{mainres} was obtained in the celebrated work by Coifman, Rochberg and Weiss \cite{CRW1976}.

The necessity of $BMO$, expressed in part (ii), was obtained in \cite{CRW1976} under the assumption that $[b,R_j]$ is bounded on $L^p$ for every Riesz transform~$R_j$.
Later this assumption was relaxed in the works by Janson \cite{J1978} and Uchiyama \cite{U1978}. It was shown there that the boundedness of $[b,T_{\O}]$ on $L^p$
(where $T_{\O}$ is defined as in Remark \ref{expl} with $\O\in C^{\infty}(S^{n-1})$ in \cite{J1978} and $\O$ is Lipschitz continuous in \cite{U1978}) implies $b\in BMO$.

Our novel points in part (ii) (even in the unweighted case and when $m=1$) are a much wider class of operators (which includes, for instance, a class of rough singular integrals)
and the fact that the restricted strong type $(p,p)$ of $[b,T_{\O}]$ (instead of the usual strong type $(p,p)$ in \cite{CRW1976,J1978,U1978}) implies $b\in BMO$.

\vskip 3mm
\item Assume that $m=1$ and $\la,\mu\in A_p$. In the one-dimensional case this result was obtained by Bloom \cite{B1985}. Recently it was extended to higher dimensions by
Holmes, Lacey and Wick~\cite{HLW2017}. Later, a quantitative form of this statement, expressed in estimate (\ref{itcase}), was obtained by the authors in \cite{LORRArxiv}.

As in the unweighted case, part (ii) is new in such generality. In \cite{HLW2017} this part was obtained, similarly to \cite{CRW1976}, assuming that $[b,R_j]$ is bounded from $L^p(\mu)$ to $L^p(\la)$ for every
Riesz transform $R_j$.

\vskip 3mm
\item
Assume that $m\ge 2$. In the unweighted setting the necessity of $BMO$ for the Hilbert transform has been recently established by Accomazzo, Parissis and P\'erez \cite{APPPc}. 

Suppose now that $\la,\mu\in A_p$. In the early 90s, Garc\'ia-Cuerva, Harboure, Segovia
and Torrea~\cite{GCHST1991} proved for a class of strongly singular integrals $S$ that $b\in BMO_{\nu^{1/m}}$ implies
$S_{b}^{m}:L^{p}(\mu)\to L^{p}(\lambda)$. It was pointed out in \cite{GCHST1991} that similar methods can be used to obtain
the corresponding estimates for Calder\'on-Zygmund operators.

Estimate (\ref{itcase}) represents a quantitative version of that statement. It looks like a natural extension of the case $m=1$ obtained in \cite{LORRArxiv}.
Notice, however, that it does not seem that this estimate can be deduced via a simple inductive argument.
Observe also that in the case of equal weights, (\ref{itcase}) recovers the sharp dependence on the $A_{p}$ constant established
by Chung, Pereyra and P\'erez \cite{CPP2012} for every $m\ge 1$. This indicates that the exponent $\frac{m+1}{2}$ in (\ref{itcase}) cannot be improved.

Recently, Holmes and Wick \cite{HWArxiv} obtained the $L^p(\mu)\to L^p(\la)$ boundedness of $T_b^m$ under the different assumption
$b\in BMO\cap BMO_{\nu}$ with $\nu=\left(\frac{\mu}{\lambda}\right)^{\frac{1}{p}}$. Hyt\"onen \cite{H2016} provided a simpler argument for this result based on the conjugation method.
We will show below (see Remark \ref{comp} in Section 4) that the assumption $b\in BMO_{\nu^{1/m}}$ is less restrictive than $b\in BMO\cap BMO_{\nu}$.

Part (ii) of Theorem \ref{mainres} for $m\ge 2$ is new even for the commutators of the Hilbert transform. Notice that in \cite{GCHST1991}
the necessity of $b\in BMO_{\nu^{1/m}}$ was deduced from the $L^p(\mu)\to L^p(\la)$ boundedness of the commutators of the Hardy-Littlewood maximal operator.
\end{list}

\vskip 2mm
Summarizing, our new contribution in Theorem \ref{mainres} is the following.
\begin{list}{\labelitemi}{\leftmargin=1em}
\item If $m\ge 2$, then both parts of Theorem \ref{mainres} are new.
\item If $m=1$, then part (ii) provides a much wider class of operators comparing to the previous works, both in weighted and unweighted cases.
\item In part (ii), the necessity of $BMO_{\nu^{1/m}}$ follows from the weighted restricted strong type $(p,p)$ estimates.
\end{list}

The rest of the paper is organized as follows. Section \ref{Prel} is devoted to present some needed preliminary results. In Section \ref{ProofThm} we prove Theorem \ref{mainres}. The last section contains some further comments and remarks related to Theorem \ref{mainres}.


\section{Preliminaries}\label{Prel}
\subsection{$A_{\infty}$ weights} Define the $A_{\infty}$ class of weights by $A_{\infty}=\cup_{p>1}A_p$.
We mention several well known properties of $A_{\infty}$ weights (see, e.g., \cite[Ch. 9]{G2}). First, if $w\in A_{\infty}$, then
$w$ is doubling, that is, for every $\la>1$, there is $c>0$ such that for all cubes $Q$,
\begin{equation}\label{pr1}
w(\la Q)\le cw(Q),
\end{equation}
where $\la Q$ denotes the cube with the same center as $Q$ and side length $\la$ times that of $Q$. Second, for every $0<\a<1$, there exists $0<\b<1$
such that for every cube $Q$ and every measurable set $E\subset Q$ with $|E|\ge \a|Q|$ one has
\begin{equation}\label{pr2}
w(E)\ge \b w(Q).
\end{equation}
Next, there exists $\ga>0$ such that for every cube $Q$,
$$
|\{x\in Q:w(x)\ge \ga w_Q\}|\ge \frac{1}{2}|Q|.
$$
In particular, this property implies immediately that for every cube $Q$ and for all $0<\d<1$,
\begin{equation}\label{pr3}
\frac{1}{|Q|}\int_Qw\le \frac{2^{1/\d}}{\ga}\left(\frac{1}{|Q|}\int_Qw^{\d}\right)^{1/\d}.
\end{equation}

\subsection{Sparse families and mean oscillations}
Given a cube $Q_0\subset {\mathbb R}^n$, let ${\mathcal D}(Q_0)$ denote the set of all dyadic cubes with respect to $Q_0$, that is, the cubes
obtained by repeated subdivision of $Q_0$ and each of its descendants into $2^n$ congruent subcubes.

A dyadic lattice ${\mathscr D}$ in ${\mathbb R}^n$ is any collection of cubes such that
\begin{enumerate}
\renewcommand{\labelenumi}{(\roman{enumi})}
\item
if $Q\in{\mathscr D}$, then each child of $Q$ is in ${\mathscr D}$ as well;
\item
every 2 cubes $Q',Q''\in {\mathscr D}$ have a common ancestor, i.e., there exists $Q\in{\mathscr D}$ such that $Q',Q''\in {\mathcal D}(Q)$;
\item
for every compact set $K\subset {\mathbb R}^n$, there exists a cube $Q\in {\mathscr D}$ containing $K$.
\end{enumerate}

A family of cubes ${\mathcal S}$ is called sparse if there exists $0<\alpha<1$ such that for every
$Q\in {\mathcal S}$ one can find a measurable set $E_Q\subset Q$ with $|E_Q|\ge \alpha|Q|$, and the sets
$\{E_Q\}_{Q\in {\mathcal S}}$ are pairwise disjoint.

Given a measurable function $f$ on ${\mathbb R}^n$ and a cube $Q$,
the local mean oscillation of $f$ on $Q$ is defined by
$$\o_{\la}(f;Q)=\inf_{c\in {\mathbb R}}
\big((f-c)\chi_{Q}\big)^*\big(\la|Q|\big)\quad(0<\la<1),$$
where $f^*$ denotes the non-increasing rearrangement of $f$.

By a median value of $f$ over a measurable set $E$ of positive finite measure we mean a possibly nonunique, real
number $m_f(E)$ such that
$$\max\big(|\{x\in E: f(x)>m_f(E)\}|,|\{x\in E: f(x)<m_f(E)\}|\big)\le |E|/2.$$

Notice that, by Chebyshev's inequality,
\begin{equation}\label{trest}
\sup_{Q}\o_{\la}(f;Q)\le \frac{1}{\la}\|f\|_{BMO}\quad (0<\la<1).
\end{equation}
By a well known result due to John \cite{JOHN1965} and Str\"omberg \cite{STR1979}, the converse estimate holds as well for $\la\le \frac{1}{2}$,
thus providing an alternative characterization of $BMO$ in terms of local mean oscillations.

Similarly to (\ref{trest}), for every weight $\eta$,
$$\sup_{Q}\o_{\la}(f;Q)\frac{|Q|}{\eta(Q)}\le \frac{1}{\la}\|f\|_{BMO_{\eta}}\quad (0<\la<1).$$
We will show that assuming $\eta\in A_{\infty}$, the full analogue of the John-Str\"omberg result holds for $\la\le \la_n$.
This fact is a simple application of the following result due to the first author \cite{L2010} and stated below in the refined
form obtained by Hyt\"onen \cite{H2014}: for every measurable function $f$ on a cube $Q$, there exists a
(possibly empty) $\frac{1}{2}$-sparse family ${\mathcal S}$ of cubes from ${\mathcal D}(Q)$ such that for a.e. $x\in Q$,
\begin{equation}\label{sposc}
|f(x)-m_f(Q)|\le 2\sum_{P\in {\mathcal S}}
\o_{\frac{1}{2^{n+2}}}(f;P)\chi_{P}(x).
\end{equation}

\begin{lemma}\label{sufc} Let $\eta\in A_{\infty}$. Then
\begin{equation}\label{bmocond}
\|f\|_{BMO_{\eta}}\le c\sup_{Q}\o_{\la}(f;Q)\frac{|Q|}{\eta(Q)}\quad\Big(0<\la\le \frac{1}{2^{n+2}}\Big),
\end{equation}
where $c$ depends only on $\eta$.
\end{lemma}

\begin{proof} Since $\o_{\la}(f;Q)$ is non-increasing in $\la$, it sufficed to prove (\ref{bmocond}) for $\la=\frac{1}{2^{n+2}}$.
Let $Q$ be an arbitrary cube. Then, by (\ref{sposc}),
\begin{eqnarray*}
\int_Q|f-f_Q|dx&\le& 2\int_Q|f-m_f(Q)|dx\le 4\sum_{P\in {\mathcal S},P\subseteq Q}\o_{\frac{1}{2^{n+2}}}(f;P)|P|\\
&\le& 4\left(\sup_{P}\o_{\frac{1}{2^{n+2}}}(f;P)\frac{|P|}{\eta(P)}\right)\sum_{P\in {\mathcal S},P\subseteq Q}\eta(P).
\end{eqnarray*}
Using that ${\mathcal S}$ is sparse and applying (\ref{pr2}), we obtain
$$\sum_{P\in {\mathcal S},P\subseteq Q}\eta(P)\le c\sum_{P\in {\mathcal S},P\subseteq Q}\eta(E_P)\le c\eta(Q),$$
which, along with the previous estimate, completes the proof.
\end{proof}

We will also use the following result proved recently in \cite[Lemma 5.1]{LORRArxiv} and closely related to (\ref{sposc}):
given a dyadic lattice ${\mathscr D}$ and a sparse family ${\mathcal S}\subset {\mathscr D}$, there exists a sparse family $\tilde{\mathcal{S}}\subset {\mathscr D}$
containing $\mathcal{S}$ and such that if $Q\in\tilde{\mathcal{S}}$,
then for a.e. $x\in Q$,
\begin{equation}\label{sparses}
|f(x)-f_{Q}|\leq 2^{n+2}\sum_{P\in\tilde{\mathcal{S}},\ P\subseteq Q}\left(\frac{1}{|P|}\int_{P}|f-f_{P}|\right)\chi_{P}(x).
\end{equation}

\section{Proof of Theorem \ref{mainres}}\label{ProofThm}
\subsection{Proof of part (i)}
The proof consists in a careful generalization of the techniques used
to establish this result in the case $m=1$ in \cite{LORRArxiv}.

We rely upon the following sparse bound obtained in \cite{IFRRArxiv}: there exist $3^n$ dyadic lattices ${\mathscr D}_j$ and sparse families ${\mathcal S}_j\subset {\mathscr D}_j$ such that
\[
|T_{b}^{m}f(x)|\leq c_{n,T}\sum_{j=1}^{3^{n}}\sum_{k=0}^{m}\binom{m}{k}\sum_{Q\in\mathcal{S}_{j}}|b(x)-b_{Q}|^{m-k}\left(\frac{1}{|Q|}\int_{Q}|b-b_{Q}|^{k}|f|\right)\chi_{Q}(x).
\]
Hence it suffices to provide suitable estimates for
\[
A_{b}^{m,k}f(x)=\sum_{Q\in\mathcal{S}}|b(x)-b_{Q}|^{m-k}\left(\frac{1}{|Q|}\int_{Q}|b-b_{Q}|^{k}|f|\right)\chi_{Q}(x),
\]
where ${\mathcal S}$ is a sparse family from some dyadic lattice ${\mathscr D}$.

We start observing that, by duality,
\begin{equation}
\|A_{b}^{m,k}f\|_{L^{p}(\lambda)}\leq\sup_{\|g\|_{L^{p'}(\lambda)}=1}\sum_{Q\in\mathcal{S}}\left(\int_{Q}|g\lambda||b-b_{Q}|^{m-k}\right)\frac{1}{|Q|}\int_{Q}|b-b_{Q}|^{k}|f|.\label{eq:DualAmk}
\end{equation}
By (\ref{sparses}), there exists a sparse family $\tilde{\mathcal{S}}\subset {\mathscr D}$
containing $\mathcal{S}$ and such that if $Q\in\tilde{\mathcal{S}}$,
then for a.e. $x\in Q$,
\[
|b(x)-b_{Q}|\leq 2^{n+2}\sum_{P\in\tilde{\mathcal{S}},\ P\subseteq Q}\left(\frac{1}{|P|}\int_{P}|b-b_{P}|\right)\chi_{P}(x).
\]
From this, assuming that $b\in BMO_{\eta}$, where $\eta$ is a weight to be chosen later, we obtain
\[
|b(x)-b_{Q}|\leq 2^{n+2}\|b\|_{BMO_{\eta}}\sum_{P\in\tilde{\mathcal{S}},\ P\subseteq Q}\eta_P\chi_{P}(x).
\]
Hence,
\begin{equation}
\begin{split} & \sum_{Q\in\mathcal{S}}\left(\int_{Q}|g\lambda||b-b_{Q}|^{m-k}\right)\frac{1}{|Q|}\int_{Q}|b-b_{Q}|^{k}|f|\\
 & \leq c\|b\|_{BMO_{\eta}}^{m}\sum_{Q\in\mathcal{\tilde{S}}}\Big(\frac{1}{|Q|}\int_{Q}|g\lambda|\Big(\sum_{P\in\tilde{\mathcal{S}},\ P\subseteq Q}\eta_P\chi_{P}\Big)^{m-k}\Big)\\
 & \times\Big(\frac{1}{|Q|}\int_{Q}\Big(\sum_{P\in\tilde{\mathcal{S}},\ P\subseteq Q}\eta_P\chi_{P}\Big)^{k}|f|\Big)|Q|.
\end{split}
\label{eq:Int}
\end{equation}

Now we notice that since the cubes from $\tilde{\mathcal S}$ are dyadic, for every $l\in {\mathbb N}$,
\begin{eqnarray*}
\left(\sum_{P\in\tilde{\mathcal{S}},\ P\subseteq Q}\eta_P\chi_{P}\right)^{l}&=&\sum_{P_1,P_2,\dots,P_l\subseteq Q,\, P_i\in {\tilde{\mathcal S}}}\eta_{P_1}\eta_{P_2}\dots \eta_{P_l}\chi_{P_1\cap P_2\cap\dots\cap P_l}\\
&\le& l!\sum_{P_l\subseteq P_{l-1}\subseteq\dots\subseteq P_1\subseteq Q,\, P_i\in {\tilde{\mathcal S}}}\eta_{P_1}\eta_{P_2}\dots \eta_{P_l}\chi_{P_l}.
\end{eqnarray*}
Therefore,
\[
\int_{Q}|h|\Big(\sum_{P\in\tilde{\mathcal{S}},\ P\subseteq Q}\eta_P\chi_{P}\Big)^{l}\leq l!\sum_{P_{l}\subseteq P_{l-1}\subseteq\dots\subseteq P_{1}\subseteq Q,\,P_{i}\in\tilde{\mathcal{S}}}\eta_{P_{1}}\eta_{P_{2}}\dots\eta_{P_{l}}|h|_{P_{l}}|P_{l}|.
\]
Further,
\[
\begin{split} & \sum_{P_{l}\subseteq P_{l-1}\subseteq\dots\subseteq P_{1}\subseteq Q,\,P_{i}\in\tilde{\mathcal{S}}}\eta_{P_{1}}\eta_{P_{2}}\dots\eta_{P_{t}}|h|_{P_{l}}|P_{l}|\\
 & =\sum_{P_{l-1}\subseteq\dots\subseteq P_{1}\subseteq Q,\,P_{i}\in\tilde{\mathcal{S}}}\eta_{P_{1}}\eta_{P_{2}}\dots\eta_{P_{l-1}}\sum_{P_{l}\subseteq P_{l-1},P_{l}\in\tilde{\mathcal{S}}}|h|_{P_{l}}\int_{P_{l}}\eta.\\
 & \leq\sum_{P_{l-1}\subseteq\dots\subseteq P_{1}\subseteq Q,\,P_{i}\in\tilde{\mathcal{S}}}\eta_{P_{1}}\eta_{L_{2}}\dots\eta_{P_{l-1}}\int_{P_{l-1}}A_{\tilde{\mathcal{S}}}(|h|)\eta.\\
 & =\sum_{P_{l-1}\subseteq\dots\subseteq P_{1}\subseteq Q,\,P_{i}\in\tilde{\mathcal{S}}}\eta_{P_{1}}\eta_{L_{2}}\dots\eta_{P_{l-1}}\left(A_{\mathcal{\tilde{S}},{\eta}}|h|\right)_{P_{l-1}}|P_{l-1}|,
\end{split}
\]
where $A_{\mathcal{\tilde{S}},{\eta}}h=A_{\tilde{\mathcal{S}}}(h)\eta$ and $A_{\tilde{\mathcal{S}}}(h)=\sum_{Q\in \tilde{\mathcal{S}}}h_Q\chi_Q$.
Iterating this argument, we conclude that
\[
\int_{Q}|h|\Big(\sum_{P\in\tilde{\mathcal{S}},\ P\subseteq Q}\eta_P\chi_{P}\Big)^{l}\lesssim\int_{Q}A^l_{\mathcal{\tilde{S}},{\eta}}|h|,
\]
where $A^l_{\mathcal{\tilde{S}},{\eta}}$ denotes the operator $A_{\mathcal{\tilde{S}},{\eta}}$ iterated $l$ times.
From this we obtain that the right-hand side of \ref{eq:Int} is controlled by
\begin{eqnarray*}
&&c\|b\|_{BMO_{\eta}}^{m}\sum_{Q\in \mathcal{\tilde{S}}}\left(\frac{1}{|Q|}\int_{Q}
A^k_{\mathcal{\tilde{S}},{\eta}}(|f|)\right)\left(\frac{1}{|Q|}\int_{Q}A^{m-k}_{\mathcal{\tilde{S}},{\eta}}(|g|\lambda)\right)|Q|\\
&&=c\|b\|_{BMO_{\eta}}^{m}\int_{{\mathbb R}^n}A_{\tilde{\mathcal{S}}}\big(A^k_{\mathcal{\tilde{S}},{\eta}}(|f|)\big)A^{m-k}_{\mathcal{\tilde{S}},{\eta}}(|g|\lambda).
\end{eqnarray*}

Using that the operator $A_{\tilde{\mathcal{S}}}$ is self-adjoint, we proceed as follows:
\begin{eqnarray*}
&&\int_{{\mathbb R}^n}A_{\tilde{\mathcal{S}}}\big(A^k_{\mathcal{\tilde{S}},{\eta}}(|f|)\big)A^{m-k}_{\mathcal{\tilde{S}},{\eta}}(|g|\lambda)=
\int_{{\mathbb R}^n}A_{\tilde{\mathcal{S}}}\big(A^k_{\mathcal{\tilde{S}},{\eta}}(|f|)\big)A_{\tilde{\mathcal{S}}}\big(A^{m-k-1}_{\mathcal{\tilde{S}},{\eta}}(|g|\lambda)\big)\eta\\
&&=\int_{{\mathbb R}^n}A_{\tilde{\mathcal{S}}}\big(A_{\tilde{\mathcal{S}}}\big(A^k_{\mathcal{\tilde{S}},{\eta}}(|f|)\big)\eta\big) A^{m-k-1}_{\mathcal{\tilde{S}},{\eta}}(|g|\lambda)
=\int_{{\mathbb R}^n}A_{\tilde{\mathcal{S}}}\big(A^{k+1}_{\mathcal{\tilde{S}},{\eta}}(|f|)\big) A^{m-k-1}_{\mathcal{\tilde{S}},{\eta}}(|g|\lambda)\\
&&=\dots=\int_{{\mathbb R}^n}A_{\tilde{\mathcal{S}}}\big(A^{m}_{\mathcal{\tilde{S}},{\eta}}(|f|)\big)|g|\lambda.
\end{eqnarray*}
Combining the obtained estimates with \eqref{eq:DualAmk} yields
\begin{equation}\label{AAm}
\|A_{b}^{m,k}f\|_{L^{p}(\lambda)}\lesssim\|b\|_{BMO_{\eta}}^{m}\|A_{\tilde{\mathcal{S}}}\big(A^{m}_{\mathcal{\tilde{S}},{\eta}}(|f|)\big)\|_{L^{p}(\lambda)}.
\end{equation}
Applying that
$\|A_{\tilde{\mathcal S}}\|_{L^p(w)}\lesssim [w]_{A_p}^{\max\big\{1,\frac{1}{p-1}\big\}}$ (see, e.g., \cite{CUMP2012}), we obtain
\begin{eqnarray*}
\|A_{\tilde{\mathcal{S}}}\big(A^{m}_{\mathcal{\tilde{S}},{\eta}}(|f|)\big)\|_{L^{p}(\lambda)}&\lesssim& [\lambda]_{A_{p}}^{\max\left\{ 1,\frac{1}{p-1}\right\}}
\|A^{m}_{\mathcal{\tilde{S}},{\eta}}(|f|)\|_{L^{p}(\lambda)}\\
&=&[\lambda]_{A_{p}}^{\max\left\{ 1,\frac{1}{p-1}\right\}}\|A_{\mathcal{\tilde{S}}}\big(A^{m-1}_{\mathcal{\tilde{S}},{\eta}}(|f|)\big)\|_{L^{p}(\lambda \eta^p)}\\
&\lesssim&\big([\lambda]_{A_{p}}[\lambda \eta^p]_{A_{p}}\big)^{\max\left\{ 1,\frac{1}{p-1}\right\}}\|A^{m-1}_{\mathcal{\tilde{S}},{\eta}}(|f|)\|_{L^{p}(\lambda \eta^p)}\\
&\lesssim&\big([\lambda]_{A_{p}}[\lambda \eta^p]_{A_{p}}[\lambda \eta^{2p}]_{A_{p}}\dots [\lambda \eta^{mp}]_{A_{p}}\big)^{\max\left\{ 1,\frac{1}{p-1}\right\}}\|f\|_{L^p(\lambda\eta^{mp})}.
\end{eqnarray*}
Hence, setting $\eta=\nu^{1/m}$, where $\nu=(\mu/\lambda)^{1/p}$ and applying (\ref{AAm}), we obtain
$$
\|A_{b}^{m,k}f\|_{L^{p}(\lambda)}\lesssim\|b\|_{BMO_{\nu^{1/m}}}^{m}
\left([\lambda]_{A_{p}}[\mu]_{A_{p}}\prod_{i=1}^{m-1}[\lambda^{1-\frac{i}{m}}\mu^{\frac{i}{m}}]_{A_{p}}\right)^{\max\left\{ 1,\frac{1}{p-1}\right\}}\|f\|_{L^p(\mu)}.
$$

By H\"older's inequality,
$$
\prod_{i=1}^{m-1}[\lambda^{1-\frac{i}{m}}\mu^{\frac{i}{m}}]_{A_{p}}\le \prod_{i=1}^{m-1}[\lambda]_{A_p}^{1-\frac{i}{m}}[\mu]_{A_p}^{\frac{i}{m}}=([\lambda]_{A_p}[\mu]_{A_p})^{\frac{m-1}{2}},
$$
which, along with the previous estimate, yields
$$
\|A_{b}^{m,k}f\|_{L^{p}(\lambda)}\lesssim\|b\|_{BMO_{\nu^{1/m}}}^{m}
\left([\lambda]_{A_{p}}[\mu]_{A_{p}}\right)^{\frac{m+1}{2}\max\left\{ 1,\frac{1}{p-1}\right\} }\|f\|_{L^{p}(\mu)},
$$
and therefore the proof of part (i) is complete.

\subsection{Proof of part (ii)} Since $\mu,\la\in A_p$, by H\"older's inequality, it follows that $\nu^{1/m}\in A_2$. Therefore, by
Lemma \ref{sufc}, it suffices to show that there exists $c>0$ such that for all $Q$,
\begin{equation}\label{sufsh}
\o_{\frac{1}{2^{n+2}}}(b;Q)\le c(\nu^{1/m})_Q.
\end{equation}

The proof of (\ref{sufsh}) is based on the following auxiliary statement.

\begin{prop}\label{auxst}
There exist $0<\e_0, \xi_0<1$ and $k_0>1$ depending only on $\Omega$ and $n$ such that the following holds.
For every cube $Q\subset {\mathbb R}^n$, there exist measurable sets $E\subset Q, F\subset k_0Q$ and $G\subset E\times F$ with
$|G|\ge \xi_0|Q|^2$ such that
\begin{enumerate}
\renewcommand{\labelenumi}{(\roman{enumi})}
\item $\o_{\frac{1}{2^{n+2}}}(b;Q)\le |b(x)-b(y)|$ for all $(x,y)\in E\times F$;
\item $\Omega\Big(\frac{x-y}{|x-y|}\Big)$ and $b(x)-b(y)$ do not change sign in $E\times F$;
\item $\Big|\Omega\Big(\frac{x-y}{|x-y|}\Big)\Big|\ge \e_0$ for all $(x,y)\in G$.
\end{enumerate}
\end{prop}

Let us show first how to prove (\ref{sufsh}) using this proposition. Combining properties (i) and (iii) yields
$$\o_{\frac{1}{2^{n+2}}}(b;Q)^m|G|\le \frac{1}{\e_0}\iint_{G}|b(x)-b(y)|^m\left|\Omega\left(\frac{x-y}{|x-y|}\right)\right|dxdy.$$
From this, and using also that $|x-y|\le \frac{k_0+1}{2}\text{diam}\,Q$ for all $(x,y)\in G$, we obtain
$$
\o_{\frac{1}{2^{n+2}}}(b;Q)^m|G|\le \frac{1}{\e_0}\Big(\frac{k_0+1}{2}\sqrt n\Big)^n|Q|\iint_{G}|b(x)-b(y)|^m\left|\Omega\left(\frac{x-y}{|x-y|}\right)\right|\frac{dxdy}{|x-y|^n}.
$$
By property (ii), $\big(b(x)-b(y)\big)^m\Omega\left(\frac{x-y}{|x-y|}\right)$ does not change sign in $E\times F$.
Hence, taking also into account that $|G|\ge \xi_0|Q|^2$, we obtain
\begin{eqnarray*}
\o_{\frac{1}{2^{n+2}}}(b;Q)^m&\le& \frac{1}{\e_0\xi_0}\Big(\frac{k_0+1}{2}\sqrt n\Big)^n\frac{1}{|Q|}\int_{E}\int_F|b(x)-b(y)|^m\left|\Omega\left(\frac{x-y}{|x-y|}\right)\right|\frac{dydx}{|x-y|^n}\\
&=&\frac{1}{\e_0\xi_0}\Big(\frac{k_0+1}{2}\sqrt n\Big)^n\frac{1}{|Q|}\int_{E}\left|\int_F\big(b(x)-b(y)\big)^m\Omega\left(\frac{x-y}{|x-y|}\right)\frac{dy}{|x-y|^n}\right|dx.
\end{eqnarray*}
Observing that $(T_{\O})_b^m$ is represented as
$$(T_{\O})_b^mf(x)=\int_{{\mathbb R}^n}\big(b(x)-b(y)\big)^m\Omega\left(\frac{x-y}{|x-y|}\right)f(y)\frac{dy}{|x-y|^n}\quad(x\not\in\text{supp}\,f),$$
the latter estimate can be written as
\begin{equation}\label{short}
\o_{\frac{1}{2^{n+2}}}(b;Q)^m\le \frac{c}{|Q|}\int_E|(T_{\O})_b^m(\chi_{F})|dx,
\end{equation}
where $c$ depends only on $\Omega$ and $n$.

By H\"older's inequality,
$$
\frac{1}{|Q|}\int_{E}|(T_{\O})_b^m(\chi_{F})|dx
\le \frac{1}{|Q|}\left(\int_{E}|(T_{\O})_b^m(\chi_{F})|^p\la dx\right)^{1/p}\left(\int_Q\la^{-\frac{1}{p-1}}\right)^{1/p'}.
$$
Using the main assumption on $T_{\O}$ along with the facts that $F\subset k_0Q$ and $\mu\in A_p$ and taking into account (\ref{pr1}), we obtain
$$\left(\int_{E}|(T_{\O})_b^m(\chi_{F})|^p\la dx\right)^{1/p}\le c\mu(F)^{1/p}\le c\mu(Q)^{1/p},$$
which, along with the previous estimate and (\ref{short}), implies
$$
\o_{1/2^{n+2}}(b;Q)^m\le c\left(\frac{1}{|Q|}\int_Q\mu\right)^{1/p}\left(\frac{1}{|Q|}\int_Q\la^{-\frac{1}{p-1}}\right)^{1/p'}.
$$

By (\ref{pr3}), $\frac{1}{|Q|}\int_Q\mu\le c\Big(\frac{1}{|Q|}\int_Q\mu^{1/r}\Big)^r$ for $r>1$. Further, by H\"older's inequality,
$$\left(\frac{1}{|Q|}\int_Q\mu^{1/r}\right)^r\le \left(\frac{1}{|Q|}\int_Q\nu^{1/m}\right)^{mp}\left(\frac{1}{|Q|}\int_I\la^{\frac{1}{r-mp}}\right)^{r-mp}.$$
Therefore, taking $r=mp+1$, we obtain
\begin{eqnarray*}
\o_{1/2^{n+2}}(b;Q)^m&\le& c\left(\frac{1}{|Q|}\int_Q\nu^{1/m}\right)^{m}\left(\frac{1}{|Q|}\int_Q\la\right)^{1/p}\left(\frac{1}{|Q|}\int_Q\la^{-\frac{1}{p-1}}\right)^{1/p'}\\
&\le& c\left(\frac{1}{|Q|}\int_Q\nu^{1/m}\right)^{m},
\end{eqnarray*}
which proves (\ref{sufsh}).

\begin{proof}[Proof of Proposition \ref{auxst}]
Let $\Sigma\subset S^{n-1}$ be an open set such that $\Omega$ does not change sign and not equivalent to zero there. Then there exists a
point $\theta_0\in\Sigma$ of approximate continuity (see, e.g., \cite[p. 46]{EG1992} for this notion) of $\Omega$ and such that $|\Omega(\theta_0)|=2\epsilon_0$
for some $\epsilon_0>0$. By the definition of approximate continuity, for every $\e>0$,
$$\lim_{\d\to 0}\frac{\sigma\{\theta\in B(\theta_0,\d)\cap S^{n-1}:|\Omega(\theta)-\Omega(\theta_0)|<\e\}}{\sigma\{B(\theta_0,\d)\cap S^{n-1}\}}=1,$$
where $B(\theta_0,\d)$ denotes the open ball centered at $\theta_0$ of radius $\d$, and $\sigma$ denotes the surface measure on $S^{n-1}$.
Therefore, for every $0<\a<1$, one can find $\d_\a>0$ such that
$$B(\theta_0,\d_\a)\cap S^{n-1}\subset \Sigma$$
and
\begin{equation}\label{intst}
\sigma\{\theta\in B(\theta_0,\d_\a)\cap S^{n-1}:|\Omega(\theta)|\ge\e_0\}\ge (1-\a)\si\{B(\theta_0,\d_\a)\cap S^{n-1}\}.
\end{equation}

Let $Q\subset {\mathbb R}^n$ be an arbitrary cube. Take the smallest $r>0$ such that $Q\subset B(x_0,r)$.
Let $\theta\in B(\theta_0,\d_{\a}/2)\cap S^{n-1}$ and let $y=x_0+R\theta$, where $R>0$ will be chosen later. Our goal is to choose $R$ such that the estimate
$\Big|\frac{x-y}{|x-y|}-\theta_0\Big|<\d_\a$ will hold for all $x\in B(x_0,r)$.

Write $x\in B(x_0,r)$ as $x=x_0+\ga\nu$, where $\nu\in S^{n-1}$ and $0<\ga<r$. We have
$$\frac{x-y}{|x-y|}=\theta+\frac{\ga\nu-(R-|x-y|)\theta}{|x-y|}.$$
Further,
\begin{eqnarray*}
\Big|\frac{\ga\nu-(R-|x-y|)\theta}{|x-y|}\Big|&\le& \frac{\ga}{|x-y|}+\frac{|R-|x-y||}{|x-y|}\\
&\le& \frac{2\ga}{|x-y|}\le \frac{2\ga}{R-\ga}\le \frac{2r}{R-r}.
\end{eqnarray*}
For every $R\ge \frac{(4+\d_\a)r}{\d_\a}$ we have $\frac{2r}{R-r}\le \frac{\d_\a}{2}$ and therefore,
$$
\Big|\frac{x-y}{|x-y|}-\theta_0\Big|\le |\theta-\theta_0|+\frac{2r}{R-r}<\d_\a.
$$

Hence, setting
$${\mathcal F}_{\a}=\Big\{x_0+R\theta: \theta\in B(\theta_0,\d_\a/2)\cap S^{n-1}, \frac{(4+\d_\a)r}{\d_\a}\le R\le \frac{(4+\d_\a)2r}{\d_\a}\Big\},$$
we obtain that
\begin{equation}\label{inclus}
\frac{x-y}{|x-y|}\in B(\theta_0,\d_\a)\cap S^{n-1}\subset \Sigma\quad((x,y)\in Q\times {\mathcal F}_{\a}).
\end{equation}
Also, it follows easily from the definition of ${\mathcal F}_{\a}$ that
\begin{equation}\label{embk}
{\mathcal F}_{\a}\subset k(\d_{\a},n)Q\quad\text{and}\quad |{\mathcal F}_{\a}|\ge \rho_n\frac{|Q|}{\d_{\a}}.
\end{equation}

By (\ref{inclus}), $\O\Big(\frac{x-y}{|x-y|}\Big)$ does not change sign on $Q\times {\mathcal F}_{\a}$. Let us show now that choosing $\a$ small enough,
we obtain that $\Big|\O\Big(\frac{x-y}{|x-y|}\Big)\Big|<\e_0$ on a small subset of $Q\times {\mathcal F}_{\a}$.
Set
$$N=\{\theta\in B(\theta_0,\d_\a)\cap S^{n-1}:|\Omega(\theta)|<\e_0\}$$
and
$$
{\mathcal G}_{\a}=\Big\{(x,y)\in Q\times {\mathcal F}_{\a}:\frac{x-y}{|x-y|}\in N\Big\}.
$$

Let us estimate $|{\mathcal G}_{\a}|$. For $x\in Q$ denote
$$
{\mathcal G}_{\a}(x)=\Big\{y\in {\mathcal F}_{\a}:\frac{x-y}{|x-y|}\in N\Big\}.
$$
Notice that by (\ref{intst}),
$$\si(N)\le \a\si(B(\theta_0,\d_\a)\cap S^{n-1})\le c_n\a\d_{\a}^{n-1}.$$
Next, for all $(x,y)\in Q\times {\mathcal F}_{\a}$ we have $|x-y|\le c'_n\frac{r}{\d_{\a}}$, and hence,
$$|{\mathcal G}_{\a}(x)|\le \Big|\Big\{s\theta: 0\le s\le c'_n\frac{r}{\d_{\a}},\theta\in N\Big\}\Big|\le c''_n\frac{|Q|}{\d_{\a}^{n}}\si(N)\le \b_n\a\frac{|Q|}{\d_{\a}}. $$
Therefore,
$$|{\mathcal G}_{\a}|=\int_Q|{\mathcal G}_{\a}(x)|dx\le \b_n\a\frac{|Q|^2}{\d_{\a}}.$$
Combining this with the second part of (\ref{embk}), we obtain that there exists $\a_0<1$ depending only on $n$ such that
\begin{equation}\label{2n}
|{\mathcal G}_{\a_0}|\le \frac{1}{2^{n+5}}|{\mathcal F}_{\a_0}||Q|.
\end{equation}

By the definition of $\o_{1/2^{n+2}}(b;Q)$, there exists a subset ${\mathcal E}\subset Q$ with $|{\mathcal E}|=\frac{1}{2^{n+2}}|Q|$ such that for every $x\in {\mathcal E}$,
\begin{equation}\label{proper1}
\o_{1/2^{n+2}}(b;Q)\le |b(x)-m_b({\mathcal F}_{\a_0})|.
\end{equation}
Next, there exist subsets $E\subset {\mathcal E}$ and $F\subset {\mathcal F}_{\a_0}$ such that $|E|=\frac{1}{2^{n+3}}|Q|$ and $|F|=\frac{1}{2}|{\mathcal F}_{\a_0}|$,
and, moreover,
\begin{equation}\label{proper2}
|b(x)-m_b({\mathcal F}_{\a_0})|\le |b(x)-b(y)|
\end{equation}
for all $x\in E, y\in F$ and $b(x)-b(y)$ does not change sign in $E\times F$. Indeed, take $E$ as a subset of
either $$E_1=\{x\in {\mathcal E}:b(x)\ge m_b({\mathcal F}_{\a_0})\}\quad\text{or}\quad E_2=\{x\in {\mathcal E}:b(x)\le m_b({\mathcal F}_{\a_0})\}$$
with $|E_i|\ge \frac{1}{2}|{\mathcal E}|$, and the corresponding $F$ will be either $\{y\in {\mathcal F}_{\a}:b(y)\le  m_b({\mathcal F}_{\a_0})\}$ with $|F|=\frac{1}{2}|{\mathcal F}_{\a_0}|$
or its complement.

Combining (\ref{proper1}) and (\ref{proper2}) yields property (i) of Proposition \ref{auxst}. Also, since
$\O\Big(\frac{x-y}{|x-y|}\Big)$ does not change sign on $Q\times {\mathcal F}_{\a_0}$, we have that property (ii) holds as well.
Next, setting $G=(E\times F)\setminus {\mathcal G}_{\a_0}$, we obtain, by the second part of (\ref{embk}) and (\ref{2n}), that
$$|G|\ge |E||F|-|{\mathcal G}_{\a_0}|\ge \frac{1}{2^{n+5}}|{\mathcal F}_{\a_0}||Q|\ge \nu_0|Q|^2,$$
where $\nu_0$ depends only on $\Omega$ and $n$, and, moreover, property (iii) follows from the definition of ${\mathcal G}_{\a_0}$.
Finally, notice that by the first part of (\ref{embk}), $F\subset {\mathcal F}_{\a_0}\subset k_0Q$
with $k_0=k(\d_{\a_0},n)$. Therefore, Proposition \ref{auxst} is completely proved.
\end{proof}

\section{Remarks and complements}
\begin{remark} The second part of Theorem \ref{mainres} leaves an interesting question whether the assumption on $\Omega$
that it does not change sign on some open subset from  $S^{n-1}$ can be further relaxed.
In particular, one can ask whether this part holds for arbitrary measurable function $\Omega$, which is not equivalent to zero.
\end{remark}

\begin{remark}\label{factor}
Similar to \cite{CRW1976,HLW2017,U1978}, Theorem \ref{mainres} can be applied to provide a weak factorization result for Hardy spaces.
For example, following Holmes, Lacey and Wick \cite{HLW2017}, one can characterize the weighted Hardy space $H^1(\nu)$ but in terms of a single singular integral,
as this was done by Uchiyama \cite{U1978}. To be more precise, under the hypotheses and notation of Theorem \ref{mainres} and for the class of operators $T_{\O}$ described in Remark \ref{expl},
we have
$$\|f\|_{H^1(\nu)}\simeq\inf\Big\{\sum_{i=1}^{\infty}\|g_i\|_{L^{p'}(\la^{1-p'})}\|h_i\|_{L^p(\mu)}: f=\sum_{i=1}^{\infty}\big(g_i(T_{\O})h_i-h_i(T_{\O})^*g_i\big)\Big\}.$$
This can be proved exactly as Corollary 1.4 in \cite{HLW2017}.
\end{remark}

\begin{remark}\label{restr}
Comparing both parts of Theorem \ref{mainres}, for the class of operators described in Remark \ref{expl} we have that the $L^p(\mu)\to L^p(\la)$ boundedness of $(T_{\O})_b^m$
is equivalent to the restricted $L^p(\mu)\to L^p(\la)$ boundedness. It is interesting that $BMO_{\nu^{1/m}}$ does not appear in this statement, though it plays the central role in the proof.
\end{remark}

\begin{remark}\label{iter}
Theorem \ref{mainres} answers the following question: what is the relation between the boundedness properties of commutators of different order?
Again, let $T_{\O}$ be a singular integral as in Remark \ref{expl}. Assume that $w\in A_p$. Then Theorem \ref{mainres} implies immediately that
for every fixed $k,m\in {\mathbb N}, k\not=m,$
\begin{equation}\label{impl}
(T_{\O})_b^m:L^p(w)\to L^p(w)\Leftrightarrow (T_{\O})_b^k:L^p(w)\to L^p(w).
\end{equation}
As in the previous remark, this implication is linked by $BMO$.

However, in the case of different weights, an analogue of (\ref{impl}) is not true in any direction, as the following example shows.

\begin{example}\label{ex} Let $n=1$ and let $H$ be the Hilbert transform. Set $\mu=|x|^{1/2}$ and $\lambda=1$. Then we obviously have that $\mu,\lambda\in A_2$.
Define $\nu=(\mu/\lambda)^{1/2}=|x|^{1/4}$ and let $b=\nu^{1/2}=|x|^{1/8}$. Then $b\in BMO_{\nu^{1/2}}$, since
for every interval $I\subset {\mathbb R},$
$$\frac{1}{\nu^{1/2}(I)}\int_I|\nu^{1/2}-(\nu^{1/2})_I|dx\le 2.$$
Therefore, by Theorem \ref{mainres}, $H_{b}^{2}:L^{2}(\mu)\to L^{2}$. On the other hand, taking
$I_{\e}=(0,\e)$ with $\e$ arbitrary small, we obtain
$$
\frac{1}{\nu(I_{\e})}\int_{I_{\e}}|\nu^{1/2}-(\nu^{1/2})_{I_{\e}}|dx=\frac{5}{4\e^{5/4}}\int_0^{\e}\big|x^{1/8}-\frac{8}{9}\e^{1/8}\big|dx\ge \frac{c}{\e^{1/8}}.
$$
Therefore, $b\not\in BMO_{\nu}$ and hence, by Bloom's theorem, $[b,H]:L^{2}(\mu)\not\to L^{2}$.

On the other hand, set $\mu=|x|^{-1/2}$ and $\lambda=1$. Then again $\mu,\lambda\in A_2$.
Define $\nu=(\mu/\lambda)^{1/2}=|x|^{-1/4}$ and let $b=\nu$. Then, arguing exactly as above, we obtain that $b\in BMO_{\nu}$
(and hence, $[b,H]:L^{2}(\mu)\to L^{2}$) and $b\not\in BMO_{\nu^{1/2}}$. Therefore, by Theorem \ref{mainres},
$H_{b}^{2}:L^{2}(\mu)\not\to L^{2}$.
\end{example}
\end{remark}

\begin{remark}\label{comp}
Compare the condition $b\in BMO_{\nu^{\frac{1}{m}}}$ with $b\in BMO\cap BMO_{\nu}$ from the works \cite{HWArxiv, H2016}.
First, as we mentioned before, if $\mu,\lambda\in A_p$, then $\nu=\left(\frac{\mu}{\lambda}\right)^{\frac{1}{p}}\in A_2$.

\begin{lemma}\label{Thm:EmmbedingClass}Let $u\in A_2$ and $r>1$. Then
\begin{equation}
BMO_{u}\cap BMO\subseteq BMO_{u^{\frac{1}{r}}}.\label{eq:embedding}
\end{equation}
Furthermore, the embedding (\ref{eq:embedding}) is strict, in general. Namely, for every $r>1$, there exists a weight $u\in A_{2}$
and a function $b\in BMO_{u^{\frac{1}{r}}}\setminus BMO$.
\end{lemma}

\begin{proof}
By (\ref{pr3}),
\[
\begin{split}\frac{1}{u^{\frac{1}{r}}(Q)}\int_{Q}|b(x)-b_{Q}|dx & \leq\frac{c}{u(Q)^{\frac{1}{r}}|Q|^{\frac{1}{r'}}}\int_{Q}|b(x)-b_{Q}|dx\\
 & =c\left(\frac{1}{u(Q)}\int_{Q}|b(x)-b_{Q}|dx\right)^{\frac{1}{r}}\left(\frac{1}{|Q|}\int_{Q}|b(x)-b_{Q}|dx\right)^{\frac{1}{r'}},
\end{split}
\]
from which \eqref{eq:embedding} readily follows.

To show the second part of the lemma, we use the same idea as in Example \ref{ex}. Let $u(x)=|x|^{\alpha}, 0<\alpha<n$. Then $u\in A_2$. Let $b=u^{1/r}=|x|^{\alpha/r}$.
Then $b\in BMO_{u^{1/r}}$. However, $b\not\in BMO$, since it is clear that $b$ does not satisfy the John-Nirenberg inequality.
\end{proof}

Take an integer $m\ge 2$. In accordance with Lemma \ref{Thm:EmmbedingClass}, take $u\in A_{2}$
and $b\in BMO_{u^{\frac{1}{mp}}}\setminus BMO$. Then, setting $\mu=u$ and $\lambda=1$, by Theorem \ref{mainres}
we obtain that $T_b^m:L^p(u)\to L^p$ for every $p\ge 2$. This kind of estimates is not covered in \cite{HWArxiv} due to the
fact that $b\not\in BMO$.
\end{remark}

\section*{Acknowledgements}
The first author was supported by ISF grant No. 447/16 and ERC Starting Grant No. 713927.
The second was supported by CONICET PIP 11220130100329CO, Argentina.
The third author was supported by the Basque Government through the BERC 2014-2017
program and by the Spanish Ministry of Economy and Competitiveness MINECO through BCAM Severo Ochoa excellence accreditation SEV-2013-0323 and also through the projects MTM2014-53850-P and MTM2012-30748.

\printbibliography

\end{document}